\documentclass{amsart}
\usepackage[utf8]{inputenc}
\usepackage{color, amsmath,amssymb,hyperref}
\newtheorem{theorem}{Theorem}
\newtheorem{proposition}{Proposition}
\newtheorem{lemma}{Lemma}
\newtheorem{example}{Example}
\newtheorem{definition}{Definition}

\newtheorem{corollary}{Corollary}
\title{Optimal expansions of Kakeya sequences}
\author{Anna Chiara Lai and Paola Loreti}
\date{\today}

\newcommand{\NN}{\mathbb N}
\begin{document}
\begin{abstract}
    We investigate optimal expansions of Kakeya sequences for the representation of real numbers. Expansions of Kakeya sequences generalize the expansions in non-integer bases and they display analogous redundancy phenomena. In this paper, we characterize optimal expansions of Kakeya sequences, and we provide conditions for the existence of unique expansions with respect to Kakeya sequences. 
\end{abstract}
\maketitle

\section{Introduction}
In this paper we investigate the optimality of expansions of the form 
\begin{equation}\label{kexp}
x=\sum_{i=1}^\infty c_ip_i\qquad (c_i)\in \{0,1\}^\NN \end{equation}
where $(p_i)$ is a \emph{Kakeya sequence} \cite{kakeya1,kakeya2}, namely a sequence of positive real numbers satisfying $p_i\to 0$ as $i\to+\infty$ and \begin{equation}\label{kdef}
  p_n\leq \sum_{i=n+1}^\infty p_i\qquad\text{for all }n\geq 1.  
\end{equation}
Examples of Kakeya expansions are provided by the classical $q$-expansions corresponding to the case $(p_i)=(q^{-i})$ with $q\in(0,1)$, first introduced in \cite{renyi}, and the Fibonacci expansions \cite{BKL21} corresponding to $(p_i)=(F^{-1}_i)$, where $(F_i)$ is the Fibonacci sequence. A classical result by Kakeya \cite{kakeya1,kakeya2} ensures that if $(p_i)$ is a Kakeya sequence then a real number $x$  admits at least one expansion of the form \eqref{kexp} if and only if $x\in[0,\sum_{i=1}^\infty p_i]$. A recent result by Baiocchi, Komornik and Loreti \cite{BKL21} ensures that, for a wide class of Kakeya sequences, every $x\in(0,\sum_{i=1}^\infty p_i)$ has indeed a continuum of different expansions, bringing to light redundancy phenomena already intensively investigated in the particular case of $q$-expansions \cite{EJK90,EK98,KL98,KL07,sidorov}.  So, if representations of the form \eqref{kexp} may well be redundant, it becomes relevant to ask whether some of such expansions satisfy some optimality criteria, or whether there exist Kakeya sequences admitting unique expansions. 


In particular, an expansion of the form \eqref{kexp} is optimal if, for all $n$, it minimizes the truncation error $x-\sum_{i=1}^n c_ip_i$ -- see Definition \ref{optdef} below for a precise statement. In the case of $q$-expansions, every $x\in[0,1/(q-1)]$ admits an optimal expansion, which coincides with the greedy expansion, if and only if the base $q$ belongs to a countable set of values, whose minimum is the Golden Mean \cite{DDKL12}, see Example \ref{ex1} below for further details.  Here we prove the following result.

\begin{theorem}\label{thm1}
Let $(p_i)$ be a strictly decreasing Kakeya sequence. Then every greedy expansion with respect to $(p_i)$ is also optimal if and only if there exists a nondecreasing sequence $(k_n)$ of  natural numbers such that
\begin{equation}
\label{except}
p_{n}=\sum_{i=n+1}^{n+1+k_n} p_i\qquad \text{for all $n\geq 1$.}
\end{equation}
\end{theorem}
As in the case of $q$-expansions, only greedy expansions of Kakeya sequences can be optimal -- see Proposition \ref{kakopt}, thus above the result fully characterizes the optimality of the greedy algorithm for strictly decreasing Kakeya sequences. As we remark in Section \ref{s3}, optimal $q$-expansions correspond to the case in which $(k_n)$ is constant. However, the study of general expansions of Kakeya sequences required an alternative methodology in the proofs due to the lack of shift-invariancy, making  arguments based on symbolic dynamics and on ergodic theory  hard to apply. We also anticipate to the reader that Fibonacci expansions are in general not optimal. A consequence of Theorem \ref{thm1} is the following error estimate.

\begin{corollary}\label{cor0}
Let $(p_i)$ be a strictly decreasing Kakeya sequence. Assume that  there exists a
 nondecreasing sequence $(k_n)$ of  natural numbers such that
\begin{equation}
\label{except0}
p_{n}=\sum_{i=n+1}^{n+1+k_n} p_i\qquad \text{for all $n\geq 1$.}
\end{equation}
Let $x\in[0,\sum_{i=1}^\infty p_i]$ and let $(c_i^*)$ and $(\hat c^*_{i})$ be the greedy and the lazy expansions of $x$, respectively. Define $\Theta^*_n:=x-\sum_{i=1}^n c_i^*p_i$ and $\hat \Theta^*_{n}:=x-\sum_{i=1}^n \hat c_i^*p_i$. Then any expansion $(c_i)$ of $x$ satisfies $$x-\sum_{i=1}^n c_ip_i\in [ \Theta^*_n,\hat \Theta^*_{n}]\quad \forall n\geq 1.$$
\end{corollary}

The aim of the present paper is twofold: on the one hand, we investigate optimality of Kakeya expansions, which is non-trivial since the Fibonacci expansions are not optimal. On the other hand, remarking the centrality of the Golden Mean also in the more general setting of Kakeya sequences, we investigate uniqueness phenomena.  

The paper is organized as follows. In Section \ref{s2}, we consider the \emph{greedy algorithm} for the expansions of Kakeya sequences, yielding the lexicographic greatest expansions.  In Section \ref{s2} we also investigate unique expansions with respect to Kakeya sequences: in Section \ref{s21} we present a general uniqueness criterion for Kakeya expansions, and our main results on the subject (Theorem \ref{thm1} and  Corollary \ref{cor1}) together with some examples of their application. In particular, we show that unique expansions do exist in the case of Tribonacci expansions, namely those expansions associated to the Kakeya sequence $(p_i):=(T_{i+1}^{-1})$ where $(T_i)$ is the Tribonacci sequence. In Section \ref{s41} we discuss a case which is not covered by Theorem \ref{thm2}, and we prove that there not exist unique expansions for the perturbations of the $\varphi$-expansions of the form $(p_i)=(\varphi^{-i}(1+(-1)^i\varepsilon_i))$, where $(\varepsilon_i)\subset(0,1)$. Finally in Section \ref{s42} we present the proofs of Theorem \ref{thm2} and of Corollary \ref{cor1}.  Section \ref{s3} is devoted the optimality of  expansions of Kakeya sequences. We show that only greedy expansions can be optimal, and we present two examples of application of Theorem \ref{thm1}, whose proof can be found in Section \ref{s3proofs}. 

\section{Greedy, lazy and unique expansions of Kakeya sequences}\label{s2}
Let $(p_i)$ be a Kakeya sequence. We call \emph{expansion} of a positive real number $x$ an expression of the form $x=\sum_{i=1}^\infty c_ip_i$, where $(c_i)$ is a binary sequence. When it is clear from the context, we use for brevity the term expansion also to refer to the corresponding digit sequence $(c_i)$. 
We recall that two sequences $(c_i)$ and $(c'_i)$ are lexicographically greater one than the other, namely $(c_i)>_{lex} (c'_i)$, if, given $N$ such that $c_N\neq c'_N$, we have $c_N>c_N'$.
We define the \emph{greedy (lazy) expansion} of $x$ with respect to $(p_i)$ as the lexicographically greatest (smallest) expansion of $x$. 

The next two results were showed in \cite{KL11}, see also \cite{K11} for an overview. For the reader's convenience we give Proposition \ref{p1} and Proposition \ref{p2} and their proofs.

\begin{proposition}[Greedy algorithm]\label{p1}
Let $(p_i)$ be a Kakeya sequence, and  $x\in[0,\sum_{i=1}^\infty p_i]$. Define the binary sequence $(c_i)$ 
as follows
\begin{equation}\label{ga}c_n:=\begin{cases}
1 & \text{ if } p_n+\sum_{\substack{
    i=1
}}^{n-1}c_ip_i\leq x\\
0 & \text{otherwise}.
\end{cases}
\end{equation}
Then 
$$x=\sum_{i=1}^\infty c_ip_i$$
and the above is the lexicographically greatest expansion of $x$ (with respect to $(p_i)$). 
\end{proposition}
\begin{proof}
To prove that $(c_i)$ is an expansion of $x$ we define the sequence $r_n:=x-\sum_{i=1}^n c_ip_i$. We inductively prove that $r_n\in[0,S_n]$ for all $n\geq 0$, where $S_n:=\sum_{i=n+1}^\infty p_i$.  The base of the induction is provided by  $r_0=x\in[0,S_0]$. Assume now $n>0$ and $r_{n-1}\in[0,S_{n-1}]$. If $c_n=0$ then by \eqref{ga} $x<p_{n}+\sum_{i=1}^{n-1}c_ip_i$. Then 
$r_n=x-\sum_{i=1}^{n-1}c_ip_i<p_{n}\leq \sum_{i=n+1}^\infty p_i=S_n$, because by assumption $(p_i)$ is a Kakeya sequence. On the other hand, $r_{n}=r_{n-1}\geq0$ and we deduce $r_n\in[0,S_n]$. If otherwise $c_n=1$ then, by \eqref{ga}, $r_{n}=r_{n-1}-p_n\geq 0$. On the other hand, by inductive hypothesis $r_{n}=r_{n-1}-p_n\leq S_{n-1}-p_n=S_n$ and this completes the proof of the claimed inclusion $r_n\in[0,S_n]$ for all $n$. In particular we have $r_n\to 0$ as $n\to \infty$ and, consequently, $x=\sum_{i=1}^n c_ip_i+r_n=\sum_{i=1}^\infty c_ip_i$.

It is left to check that the above is the lexicographic greatest expansion of $x$. In order to seek a contradiction assume that $x=\sum_{i=1}^\infty c'_ip_i$ and that $(c_i)<_{lex} (c'_i)$. Then there exists $N\geq 1$ such that $c_i=c'_i$ for all $1\leq i<N$  and that $c_N<c_N'$. In particular, $c_N=0$ and $c_N'=1$. By \eqref{ga} we have $$x<p_N+\sum_{i=1}^{N-1}c_ip_i=p_N+\sum_{i=1}^{N-1}c'_ip_i=\sum_{i=1}^{N}c'_ip_i=x-\sum_{i=N+1}^\infty c'_ip_i\leq x.$$
yielding the required contradiction. 
\end{proof}

 Greedy and lazy expansions can be characterized by using the above result. 
\begin{proposition}[Greedy and lazy expansions]\label{p2}
Let $(p_i)$ be a Kakeya sequence,  $x\in[0,\sum_{i=1}^\infty p_i]$ and let $(c_i)$ be an expansion of $x$. 
\begin{enumerate}
    \item[(i)] If for all $n\geq 1$ such that $c_n=0$ we have $\sum_{i=n+1}^\infty c_{i}p_i<p_n$ then $(c_i)$ is the greedy expansion of $x$.
    \item[(ii)] If for all $n\geq 1$ such that $c_n=1$ we have $\sum_{i=n+1}^\infty c_{i}p_i>\sum_{i=n+1}^\infty p_i-p_n$ then $(c_i)$ is the lazy expansion of $x$.
\end{enumerate}

\end{proposition}
\begin{proof}
Let $(c_i)$ be an expansion of $x$. Assume that condition (i) is satisfied. If $(c_i)$ is unique then it is trivially the greedy expansion of $x$. If otherwise there exists another expansion $(c_i')$ of $x$ then there is a smallest $n$ such that $c_n\not=c'_n$. In particular, if $n> 1$ then $c_i=c_i'$ for $i=1,\dots,n-1$. In order to seek a contradiction, assume that $(c'_i)$ is lexicographically greater than $(c_i)$. Therefore $c_n=0$ and $c_n'=1$. But we then have by condition (i) 
$$0\leq \sum_{i=n+1}^\infty c'_ip_i= x-\sum_{i=i}^{n-1}c'_ip_i-p_n=x-\sum_{i=i}^{n-1}c_ip_i-p_n=\sum_{i=n+1}^\infty c_ip_i-p_n<0$$
and, consequently, the required contradiction. 

Consider now the reflected sequence $(\bar c_i)$ defined by $\bar c_n:=1-c_n$ for all $n\geq 1$. Assume  condition (ii) and  note that this implies if $c_n=1$ then $(\bar c_i)=0$ and $\sum_{i=n+1}^\infty \bar c_ip_i<p_n$. In other words, $(\bar c_i)$ satisies condition (i) and, in view of the above arguments, is the greedy expansion of $\bar x:=\sum_{i=1}^\infty p_i-x$. Suppose now that $(c'_i)$ is a lexicographically smaller expansion of $x$. Then $(\bar c'_i)$ is both an expansion of $\bar x$ and it is lexicographically greater than $(\bar c_i)$, which is a contradiction with the fact that $(\bar c_i)$ is greedy.
  \end{proof}
  \subsection{Unique expansions of Kakeya sequences}\label{s21}
We present some sufficient and some necessary conditions for the existence of unique expansions with respect to a strictly decreasing Kakeya sequence.

Adapting classical arguments for $q$-expansions, we have the following criterion for the uniqueness. 

\begin{proposition}[Unique expansions]\label{p2u}
Let $(p_i)$ be a Kakeya sequence, $(c_i)\in\{0,1\}^\infty$ and  $x\in[0,\sum_{i=1}^\infty p_i]$ defined by $x:=\sum_{i=1}^\infty c_ip_i$. Then $(c_i)$ is the unique expansion of $x$ if and only if
\begin{enumerate}
    \item[(i)] For all $n\geq 1$ such that $c_n=0$ we have $\sum_{i=n+1}^\infty c_{i}p_i<p_n$.
    \item[(ii)] For all $n\geq 1$ such that $c_n=1$ we have $\sum_{i=n+1}^\infty c_{i}p_i>\sum_{i=n+1}^\infty p_i-p_n$.
\end{enumerate}
Moreover if $(c_i)\not=(0)^\infty,(1)^\infty$ and if $(c_i)$ is unique, then it is infinite, that is it cannot end with either $(0)^\infty$ or $(1)^\infty$.
\end{proposition}
\begin{proof}
The sequence $(c_i)$ is unique if and only if it is simultaneously the greedy and the lazy expansion of $x$. Therefore the first part of the claim follows by Proposition \ref{p2}. Now, assume that $(c_i)$ is unique and, in order to seek a contradiction, that there is a smallest $N$ satisfying $(c_{N+i})=(1)^\infty$. Since $(c_i)\not=(1)^\infty$ then $c_N=0$. By above arguments $\sum_{i=N+1}^\infty c_ip_i=\sum_{i=N+1}^\infty p_i< p_N$.
But by assumption $(p_i)$ is a Kakeya sequence, therefore $\sum_{i=N+1}^\infty p_i\geq p_N$, implying the required contradiction. Then any unique expansion cannot end with $(1)^\infty$. Now, if $(c_i)$ is unique also its reflection $(\bar c_i)=(1-c_i)$ is unique, we can conclude that the tail of $(c_i)$ is also different from $(0)^\infty$.
\end{proof}

We present our main result on the uniqueness of expansions of Kakeya sequences, while postponing its proof to Section \ref{s42} below.

\begin{theorem}\label{thm2}
Let $(p_i)$ be a strictly decreasing Kakeya sequence. 

A sufficient condition for the existence of unique expansions with respect to $(p_i)$ is the following.
There exists an integer $N\geq 1$ such that 
\begin{equation}\label{cnu1}\sum_{i=0}^\infty p_{n+2i} < p_{n-1} \qquad \text{for all $n\geq. N$}\end{equation}
In particular, $0^N(10)^\infty$ is a unique expansion with respect to $(p_i)$.

If moreover $p_{n}\leq 2 p_{n+1}$ for all sufficiently large $n$,  necessary conditions for the existence of unique expansions with respect to $(p_i)$ are the following: 
\begin{equation}\label{csu1}\sum_{i=n+1}^\infty p_i \leq p_{n-1} \qquad \text{for all sufficiently large $n$}\end{equation}
and
\begin{equation}\label{csu2}\sum_{i=n+1}^\infty p_i < p_{n-1} \qquad \text{for infinitely many $n$.}\end{equation}

\end{theorem}
The next result proves the Golden Mean $\varphi$ to be a threshold for the existence of unique expansions not only for the $q$-expansions, but also for general  expansions of Kakeya sequences. In particular if the ratio $p_n/p_{n+1}$ between two consecutive terms of the Kakeya sequence is definitively greater than $\varphi$ then there exist unique expansions, if otherwise the ratio $p_n/p_{n+1}$ is definitively smaller or equal than $\varphi$ then there are not unique expansions with respect to $(p_i)$.
\begin{corollary}\label{cor1}
Let $(p_i)$ be a strictly decreasing Kakeya sequence.
\begin{enumerate}
    \item If there exists $N$ such that $p_n\geq \varphi p_{n+1}$ for all  $n\geq N$ and $p_n> \varphi p_{n+1}$ for infinitely many $n$,  then $0^N(10)^\infty$ is a unique expansion with respect to $(p_i)$.
\item If there exists $N$ such that $p_n\leq \varphi p_{n+1}$ for all  $n\geq N$ then  every $x\in(0,\sum_{i=1}^\infty p_i)$ admits at least two expansions with respect to $(p_i)$.
\end{enumerate}
\end{corollary}

The proof of Corollary \ref{cor1} can be found in Section \ref{s3proofs}. Note that it has been proved in \cite[Corollary 4.2]{BKL21}, that the  condition $p_n<\varphi p_{n+1}$  for all $n\geq 1$ (which is stronger than condition (2) in the above result), implies that every $x\in(0,\sum_{i=1}^\infty p_i)$ has a continuum of different expansions. 

Here we present some examples of application of the above results.

\begin{example}[$q$-expansions]
If $p_n=q^n$ with $q\in(1,2)$ then condition \eqref{cnu1} is equivalent to $q>\varphi$, where $\varphi$ is the Golden ratio. We then recover the classical result by \cite{EJK90}. Note that in this case \eqref{cnu1} and \eqref{csu2} are equivalent and we have that $q>\varphi$ is a sufficient and necessary condition for the existence of unique expansions. Also, in agreement with the classical literature $(10)^\infty$ is the minimal unique expansion. 
\end{example}

\begin{example}[Perturbations of $q$-expansions]
Let $(p_i)$ be a stricty decreasing  sequence satisfying 
$p_n=\varphi^{-n}(1+\varepsilon_n)$ for some sequence $(\varepsilon_i)$ such that $\varepsilon_n>-1$ for all $n$ and
\begin{equation}\label{perturbation}\frac{1+\inf_i \varepsilon_i}{1+\sup_i \varepsilon_i}\geq \varphi-1.\end{equation}
Then by \cite[Proposition 3.2]{BKL21}, $(p_i)$ is a Kakeya sequence. If 
$(\varepsilon_i)$ is eventually decreasing and $\varepsilon_n>\varepsilon_{n+1}$ for infinitely many $n$, then $\varphi p_{n+1}=\varphi^{-n}(1+\varepsilon_{n+1})\leq \varphi^{-n}(1+\varepsilon_{n})=p_n$ and the inequality is strict for infinitely many $n$. By Corollary \ref{cor1} we obtain that $0^N(10)^\infty$ is unique for a sufficiently large $N$. If otherwise $\varepsilon_n$ is eventually increasing then $\varphi p_{n+1}\geq p_n$ and, again by Corollary \ref{cor1} there are not unique expansions with respect to $(p_i)$. 
\end{example}

\begin{example}[Fibonacci and Tribonacci expansions]
We know from \cite[Theorem 1.5]{BKL21} that for all $x\in(0,\sum_{i=1}^\infty F^{-1}_i)$ there exists a continuum of different Fibonacci expansions. Consider now the Tribonacci sequence
$$T_{1}:=1,\quad T_{2}:=1,\quad  T_3:=2, \quad T_{n+3}:=T_n+T_{n+1}+T_{n+2}.$$
and define  $(p_i):=(T_{i+1}^{-1})$. Note that $(p_i)$ is a strictly decreasing Kakeya sequence. Indeed $(p_i)$ is clearly positive, vanishing and a simple inductive argument proves that $T_{n+1}\leq 2T_{n}$ for all $n\geq 1$. Therefore we have $p_{n}\leq 2^{-i}p_{n+i}$, for all $n,i\geq 1$. From this, in the fashion of \cite[Lemma 3.1]{BKL21}, we conclude
$$\sum_{i=n+1}^\infty p_{i}=\sum_{i=1}^\infty p_{n+i}\geq \sum_{i=1}^\infty 2^{-i}p_{n}=p_n.$$
As it is well known \cite{fibtrib}, $T_{n+1}/T_{n}\to q$ where $q=1,83928675\dots$ is the positive root of $x^3-x^2-x-1$ -- incidentally note that $q$ is a Pisot number and it belongs to the set of exceptional values $P$ for which every greedy $q$-expansion is also optimal \cite{DDKL12}, see Example \ref{ex1} below. We then have that $p_{n}/p_{n+1}\to q>\varphi$ and, consequently $p_n>\varphi p_{n+1}$ for all sufficiently large $n$. By Corollary \ref{cor1}, there exist unique Tribonacci expansions. 
\end{example}
\subsection{Other results on unique expansions}\label{s41}
By Theorem \ref{thm2}, if $\sum_{i=n+1}^\infty p_i < p_{n-1} $ for infinitely many $n$ and 
    $\sum_{i=n+1}^\infty p_i \leq p_{n-1}$ for all sufficiently large $n$ then there can be unique expansions. If the stronger condition \eqref{cnu1} is satisfied then a unique expansion is provided by $0^N(10)^\infty$ with $N$ sufficiently large. We note that the sufficient condition \eqref{cnu1} for the existence of a unique expansion directly implies
\begin{equation}\label{cnu2}\sum_{i=n+1}^\infty p_i < p_{n-1} \qquad \text{for all sufficiently large $n$.}\end{equation}
    Indeed \eqref{cnu1} implies for $n$ large enough
    $$\sum_{i=0}^\infty p_{n+2i} < p_{n-1} \quad \text{and} \quad
   \sum_{i=0}^\infty p_{n+2i+1} < p_{n}.$$
    Summing above inequalities we deduce $\sum_{i=n}^\infty p_i<p_{n-1}+p_n$ and this implies $\sum_{i=n+1}^\infty p_i<p_{n-1}$.
However there may be Kakeya sequences, not covered by the assumptions of Theorem \ref{thm2}, such that
\begin{enumerate}
    \item [1)] $\sum_{i=n+1}^\infty p_i \leq p_{n-1}$ for all sufficiently large $n$,
    \item [2)]$\sum_{i=n+1}^\infty p_i < p_{n-1} $ for infinitely many $n$, 
\end{enumerate}
   but failing to satisfy \eqref{cnu1}, because $\sum_{i=0}^\infty p_{n+2i}\geq p_{n-1}$ for infinitely many $n$. 
    
    A particular case is discussed in the following result.
    
    \begin{proposition}\label{alternatedp}
    Let $(p_i)$ be a Kakeya sequence such that {for all sufficiently large $n$}
    \begin{align}\label{alternated}&p_{2n}\geq \sum_{i=0}^\infty p_{2n+2i+1}, \qquad p_{2n+1}\leq \sum_{i=1}^\infty p_{2n+2i}.\end{align}
    
    Then every $x\in(0,\sum_{i=1}^\infty p_i)$ has at least two expansions with respect to $(p_i)$. 
    \end{proposition}
    \begin{proof}
    In order to seek a contradiction, assume that there exists $x\in(0,\sum_{i=1}^\infty p_i)$ admitting a unique expansion $(c_i)$. By Proposition \ref{p2u} the expansion $(c_i)$ is infinite and, in particular, there exist infinitely many $N$ such that $c_{2N-1}\not=c_{2N}$. Since also $(1-c_i)$ is unique, then we can assume without loss of generality that $c_{2N-1}=0$ and $c_{2N}=1$ and that \eqref{alternated} is satisfied for all $n\geq N-1$. We claim that \begin{equation}\label{lexes}
        (c_{2N+i})\leq_{lex} (01)^\infty.
    \end{equation} Suppose on the contrary that $(c_{2N+i})>_{lex}(01)^\infty$. Since $c_{2N-1}c_{2N}=01$, then for some $n\geq N-1$ one has $c_{2n+1}c_{2n+2}c_{2n+3}=011$. This, together with \eqref{alternated}, implies that 
    $\sum_{i=1}^\infty c_{2n+1+i}p_i\geq p_{2n+2}+p_{2n+3}\geq p_{2n+2}+\sum_{i=2}^\infty p_{2n+2i}=\sum_{i=1}p_{2n+2i}^\infty\geq p_{2n+1}$. On the other hand, since $c_{2n+1}=0$ then, by Proposition \ref{p2u}, $\sum_{i=1}^\infty c_{2n+1+i}p_i<p_{2n+1}$. Then assuming $(c_i)>_{lex}(01)^\infty$ implies a contradiction and we get the proof \eqref{lexes}.
    
    Now, we have that $(c_{2N+i})^\infty\not=(01)^\infty$, because otherwise one would have $c_{2N-1}=0$ and $\sum_{i=1}^\infty c_{2N-1+i}p_{2N-1+i}=\sum_{i=0}^\infty p_{2N+2i}\geq p_{2N-1}$, thus contradicting the uniqueness of $(c_i)$.
    In view of above arguments, we have $c_{2N-1}c_{2N}=01$ and $(c_{2N+i})<_{lex}(01)^\infty$.
Then there exists $n\geq N$ such that $c_{2n}c_{2n+1}c_{2n+2}=100$. We then have
\begin{align*}\sum_{i=1}^\infty c_{2n+i}p_{2n+i}&\leq \sum_{i=3}^\infty p_{2n+i}=\sum_{i=1}^\infty p_{2n+i}-p_{2n+1}-p_{2n+2}\\
&\leq \sum_{i=1}^\infty p_{2n+i}-p_{2n+1}-\sum_{i=1}^\infty p_{2n+2i+1}=\sum_{i=1}^\infty p_{2n+i}-\sum_{i=0}^\infty p_{2n+2i+1}\\
&\leq \sum_{i=1}^\infty p_{2n+i}-p_{2n}.\end{align*}
On the other hand, we assume that $(c_i)$ is unique, then in view of \eqref{p2u} $c_{2n}=1$ implies $\sum_{i=1}^\infty c_{2n+i}p_{2n+i}>\sum_{i=1}^\infty p_{2n+i}-p_{2n}$. Then $(c_i)$ cannot be unique and this concludes the proof.  
    \end{proof}
    
We provide an example of application.
\begin{example}[Alternated perturbations of $\varphi$-expansions]
    If $(p_i)$ is a Kakeya sequence such that
    $p_n=\varphi^{-n}(1+(-1)^n\varepsilon_n)$ for all $n$ and for some positive sequence $(\varepsilon_i)\subset[0,1)$ -- recall that sequences $(\varepsilon_i)$ satisfying \eqref{perturbation} ensure that $(p_i)$ is a Kakeya sequence.
We have for all $n$
    $$\sum_{i=0}^{\infty} p_{2n+2i+1}\leq\sum_{i=0}^{\infty} \frac{1}{\varphi^{2n+2i+1}}=\frac{1}{\varphi^{2n}}\leq p_{2n}$$
    and
    $$\sum_{i=0}^{\infty} p_{2n+2i}\geq \sum_{i=0}^{\infty} \frac{1}{\varphi^{2n+2i}}=\frac{1}{\varphi^{2n-1}}\geq p_{2n-1}.$$
    By Proposition \ref{alternatedp} there are not unique expansions with respect to $(p_i)$.
\end{example}
    \medskip

\subsection{Proof of Theorem \ref{thm2} and proof of Corollary \ref{cor1}}\label{s42}
\begin{proof}[Proof of Theorem \ref{thm2}]
Let $N$ be such that \eqref{cnu1} is satisfied and consider $(c_i)=0^N(10)^\infty$. 
In view of Proposition \ref{p2u}, we need to prove that for all $n\geq1$ 
\begin{equation}\label{testu1}
\begin{split}c_n=0\quad \Rightarrow \sum_{i=n+1}^\infty c_ip_i<p_n\\
c_n=1\quad \Rightarrow \sum_{i=n+1}^\infty c_ip_i>-p_n+\sum_{i=n+1}^\infty p_i.
\end{split}
\end{equation}
Now assume that $c_n=0$. Then either $n< N$  or $(c_{n+i})=(10)^\infty$. In the first case we obtain by \eqref{cnu1}
$$\sum_{i=n+1}^\infty c_ip_i=\sum_{i=N+1}^\infty c_ip_i=\sum_{i=0}^\infty p_{N+2i+1}<p_{N}\leq p_n.
$$
If otherwise $n\geq N$ then, again by \eqref{cnu1}
$$\sum_{i=n+1}^\infty c_ip_i=\sum_{i=0}^\infty p_{n+2i+1}<p_n.
$$
Finally if $c_n=1$ then $n> N$ and $(c_{n+i})=(01)^\infty$. Therefore
$$\sum_{i=n+1}^\infty c_ip_i=\sum_{i=1}^\infty p_{n+2i}=\sum_{i=n+1}^\infty p_i-\sum_{i=0}^\infty p_{n+2i+1}>\sum_{i=n+1}^\infty p_i-p_n$$
and this proves the uniqueness of $(c_i)$.

\medskip We now prove the second part of the claim. Assume that $p_n\leq 2 p_{n+1}$ for all  $n\geq N$, where $N\geq 1$. Then $p_{n+i}\geq 2^{-i}p_n$ for all $n\geq N$ and, consequently,
\begin{equation}\label{2est}
    \sum_{i=1}^{\bar n-n} p_{n+i}\geq 
    p_n\sum_{i=1}^{\bar n -n} 2^{-i}=p_n\left(1-2^{-\bar n+n})\right)\qquad \text{for all } \bar n>n\geq N.
\end{equation} First we show that if $\sum_{i=n+1}^\infty p_i > p_{n-1}$ for infinitely many $n$ then there are not unique expansions with respect to $(p_i)$, which proves that \eqref{csu1} is a necessary condition for the existence of unique expansions with respect to $(p_i)$. Fix $n\geq N$ and consider $\bar n>n$ such that $\sum_{i= \bar n+1}^\infty p_i > p_{ \bar n-1}$. We then obtain (also using \eqref{2est})  \begin{align*}
\sum_{i=1}^\infty p_{n+i}&= \sum_{i=1}^{\bar n-n-1}p_{n+i}+p_{\bar n}+\sum_{i=\bar n+1}^{\bar \infty}p_i>\sum_{i=1}^{\bar n-n-1}p_{n+i}+p_{\bar n}+p_{\bar n-1}\\
&\geq p_n\left(1-2^{-(\bar n-n-1)}\right)+p_{\bar N-1}+2^{-(\bar n-n-1)}p_n=p_n+p_{\bar n}\\
&>p_n \qquad\qquad \text{for all }n\geq N.
\end{align*}
Then the shifted Kakeya sequence $(p_{N+i})$ meets the assumptions of \cite[Theorem 1.4]{BKL21} and this implies that every $x_N\in (0,\sum_{i=1}^\infty p_{N+i})$ has a continuum of different expansions. Now, let be $(c_i)$ be a binary sequence. By Proposition \ref{p2u}, if $(c_i)$ is finite then it is not unique. If otherwise $(c_i)$ is infinite, setting $x:=\sum_{i=1}^\infty c_ip_i$, then $x_N:=x-\sum_{i=1}^N c_ip_i=\sum_{i=N+1}^\infty c_ip_i$ belongs to the set $(0,\sum_{i=1}^\infty p_{N+i})$. In view of the above arguments,  there exists a continuum of sequences $(\hat c_{N+i})$ expanding the rest $x_N$ and, consequently, each of them provides a different expansion of $x$ of the form $c_1\cdots c_N(\hat c_{N+i})$.

To prove that the uniqueness of $(c_i)$ implies \eqref{csu2} we argue as follows. By Proposition \ref{p2u} the sequence $(c_i)$ must be infinite, that is it cannot end with $(0)^\infty$ or $(1)^\infty$.  This implies that there exist infinite $n$ such that $c_n\not= c_{n+1}$. In particular, there exists an index sequence $n_j$ such that $c_{n_j}=0$ and $c_{n_j+1}=1$, for all $j\geq 1$. Then, again Proposition \ref{p2} implies that
for all $j\geq 1$
$$p_{n_j+1}+\sum_{i=n_j+2}^\infty c_ip_i=\sum_{i=n_j+1}^\infty c_ip_i<p_{n_j}$$
and, since $c_{n_j+1}=1$, then
$$\sum_{i=n_j+2}^\infty c_ip_i=\sum_{i=n_j+1}^\infty c_ip_i-p_{n_j+1}+\sum_{i=n_j+2}^\infty p_i.$$
By comparing above two inequalities we deduce
$$\sum_{i=n_j+2}p_i<p_{n_j}\qquad \text{for all }j\geq 1$$
and, consequently, \eqref{csu2}.
\end{proof}
\begin{proof}[Proof of Corollary \ref{cor1}]
 Assume first that $p_{n}\geq \varphi p_{n+1}$ for all $n\geq N$, where $N$ is a sufficiently large integer, and $p_{n}>\varphi p_{n+1}$ for infinitely many $n$. Iterating these relations we obtain $p_{n}\geq \varphi^k p_{n+k}$ for all $k\geq 1,n\geq N$ and $p_{n}> \varphi^k p_{n+k}$ for infinitely many $k$. Then we have for $n>N$
 \begin{align*}\sum_{i=0}^\infty p_{n+2i}< \sum_{i=0}^\infty p_{n-1}\varphi^{-2i-1}= p_{n-1}.\end{align*}
Then by Theorem \ref{thm2} we have that $0^N(10)^\infty$ is a unique expansion.

Assume now that $p_{n}\leq \varphi p_{n+1}$ for all $n\geq N$. Then
\begin{align*}\sum_{i=n+1}^\infty p_{i}\geq \sum_{i=2}^\infty p_{n-1}\varphi^{-i}=\frac{p_{n-1}}{\varphi(\varphi-1)}= p_{n-1}.\end{align*}
Moreover $p_n\leq 2 p_{n+1}$ and, consequently, $(p_i)$ does not meet the necessary conditions stated in Theorem \ref{thm2} for the existence of unique expansions. 
\end{proof}

\section{Optimal expansions of Kakeya sequences}\label{s3}

We begin with a definition. 

\begin{definition}[Normalized error, optimal expansions, and optimal Kakeya sequences]\label{optdef} Let $(p_i)$ be a Kakeya sequence. We define the \emph{normalized error} associated to an expansion $(c_i)$ of $x\in[0,\sum_{i=1}^\infty p_i]$ as $$\Theta_n((c_i)):=r_n\Big(x-\sum_{i=1}^n c_ip_i\Big)$$
where $r_n:=\sum_{i=1}^\infty p_i/\sum_{i=1}^np_i$. 
The expansion $(c_i^*)$ of a positive real number is \emph{optimal} if any other expansion $(c_i)$ of $x$ satisfies $$\Theta_n((c^*_i))\leq \Theta_n((d_i))\qquad \text{ for all $n
\geq 1$.}$$
Finally $(p_i)$ is an optimal Kakeya sequence if for all $x\in(0,\sum_{i=1}^\infty p_i)$ there exists an optimal expansion.
\end{definition}

Next we show that only greedy expansions can be optimal.

\begin{proposition}\label{kakopt}
Assume that $(p_i)$ is a Kakeya sequence. If $(c_i^*)$ is the optimal expansion of $x\in[0,\sum_{i=1}^\infty p_i]$ with respect to $(p_i)$, then $(c_i^*)$ is the greedy expansion of $x$.
\end{proposition}
\begin{proof}
It suffices to show that $(c^*_i)$ is the lexicographic greatest expansion of $x$. To this end, let $(c_i)\not=(c^*_i)$ be another expansion of $x$ and let $n$ be the smallest index such that $c_n\not=c^*_i$ then 
$$\Theta_n((c^*_i))\leq \Theta_n((c_i))$$
and this implies that, since $c_i=c^*_i$ for all $i=1,\dots, n-1$ then $c_n^*p_n\geq c_np_n$ which implies, together with $c_n\not=c^*_n$ that $c_n<c^*_n$ that is, that $(c^*_i)$ is lexicographically greater than $(c_i)$. By the arbitrariness of $(c_i)$ we deduce that $(c_i^*)$ is the greedy expansion of $x$.
\end{proof}

We propose two examples of application of Theorem \ref{thm1} while postponing its proof in Section \ref{s3proofs} below.

\begin{example}[$q$-expansions]\label{ex1} Let $p_n=q^{-n}$ for some $q\in (1,2)$. In this case, the condition \eqref{except} is equivalent to the existence of some integer $k$ such that 
$$1=\sum_{i=1}^{k} q^{-i}.$$
We then recover by a different proof the result in \cite[Theorem 1.3]{DDKL12}, implying that optimal $q$-expansions 
exist of all $x\in(0,\frac{1}{q-1})$ if and only if the base $q$ belongs to the exceptional set 
$$P:=\left\{q\in (1,2)\mid 1=\sum_{i=1}^n q^{-i}\, \text{ for some }n\in\NN\right\}.$$
Also note that the minimal value in $P$ is the Golden Ratio $\varphi$.\end{example} In classical theory of $q$-expansion $\varphi$ represents an important threshold: every real number in $(0,1/(q-1))$ has a continuum of different $q$-expansions if and only if $q<\varphi$ \cite{EJK90}  or, equivalently, if and only if $1/q<1/q^2+1/q^3$. This phenomenon generalizes to Kakeya expansions. 
Indeed, if $(p_i)$ is a Kakeya expansion and if $p_{n}<\sum_{i=n+1}^{n+1+k_n} p_i$ for all positive sequences of natural numbers $(k_n)$, then in particular
$p_n< p_{n+1}+p_{n+2}$ for all $n\geq 1$. 
Iterating this relation we deduce $p_n\leq \sum_{i=0}^\infty p_{n+2i+1}$ and, consequently, 
$$\sum_{i=n+2}^\infty p_i=\sum_{i=1}^\infty p_{n+2i+1}+\sum_{i=1}^\infty p_{n+2i}\geq p_{n+1}+p_{n+2}\geq p_n\qquad \forall n\geq 1.$$
Then, as showed in \cite{BKL21}, every real number $x\in (0, \sum_{i=1}^\infty p_i)$ as a continuum of different expansions with respect to $(p_i)$.

\begin{example}[Fibonacci expansions]
We consider the Fibonacci expansion
$$x=\sum_{i=1}^\infty \frac{c_i}{F_{i+1}}\qquad (c_i)\in \{0,1\}^\NN$$
where $F_i$ are the Fibonacci numbers
$$F_1:=1,\quad F_2:=1\quad F_{i+2}=F_{i+1}+F_{i}\quad i=1,2,\dots.$$ The sequence $(F^{-1}_{i+1})$ is a strictly decreasing Kakeya sequence \cite{BKL21} for which \eqref{except} does not hold. Then by Theorem \ref{thm1} there exists some real number $x\in (0,\sum_{i=1}^\infty F_{i+1}^{-1})$ that does not admit optimal expansions.  An example is provided by the number $$x=\frac{1}{F_4}+\frac{1}{F_5}=\frac{1}{3}+\frac{1}{5}=\frac{8}{15}.$$ By Proposition \ref{kakopt} it suffices to prove that the greedy expansion of $x$ is not optimal. Now, a Fibonacci expansion of $x$ is provided by $(c_i)=0011(0)^\infty$. On the other hand, the greedy algorithm yields as first digits $0100$, therefore if $(c^*_i)$ is the greedy expansion of $x$ then
$x-\sum_{i=1}^5 c_i^*F_{i}^{-1}=\frac{1}{30}$
whereas $x-\sum_{i=1}^5 c_iF_{i}^{-1}=0$
from which we deduce
$\Theta_5((c^*_i))>\Theta_5((c_i))=0.$
\end{example}

\subsection{Proof of Theorem \ref{thm1} and proof of Corollary \ref{cor0}}\label{s3proofs}
We fix a strictly decreasing Kakeya sequence $(p_i)$ and we begin with some general remarks on  \eqref{except}. If there exists a  sequence of naturals $(k_n)$ satisfying 
$$p_{n}=\sum_{i=n+1}^{n+1+k_n} p_i\qquad \text{for all $n\geq 1$,}$$
then $k_n>0$ for all $n\geq 1$ and
\begin{equation}\label{golden}
p_{n}\geq p_{n+1}+p_{n+2} \qquad \text{ for all $n\geq 1$.}
\end{equation}
We need the next technical lemma.
\begin{lemma}\label{star}
Let $(p_i)$ be a strictly decreasing Kakeya sequence for which there exists a sequence $(k_n)$ of natural numbers such that
\begin{equation}
\label{except2}
p_{n}=\sum_{i=n+1}^{n+1+k_n} p_i\qquad \text{for all $n\geq 1$.}
\end{equation}
From any fixed positive $N_0\in\NN$, recursively define the sequence $N_{h+1}:=N_{h}+1+k_{N_h}$ for $h\geq 0$. Then for all $H\geq 1$ and for all $(d_i)\in\{-1,0,1\}^\infty$ satisfying
\begin{equation}\label{el1}
\sum_{i=N_0+1}^{N_h}d_ip_i<p_{N_0} \quad \forall h=1,\dots,H,
\end{equation}
one has either the identity 
\begin{equation}\label{el2}
 \sum_{i=N_0+1}^{N_H}d_ip_i=p_{N_0}-p_{N_H}   
\end{equation}
 or the estimate 
 \begin{equation}\label{el3}
 \sum_{i=N_0+1}^{N_H}d_ip_i\leq p_{N_0}-{\min \{p_{N_H-1},2p_{N_H}\}}.\end{equation}
\end{lemma}
\begin{proof}
First of all we remark by definition, $p_{N_h}=\sum_{i=N_{h}+1}^{N_{h+1}}p_i$. Moreover since $(p_i)$ is strictly decreasing then $(k_n)$ is positive and this implies $N_{h-1}+2\leq N_h$ for all $h\geq 1$. 
We prove the claim by induction on $H$. If $H=1$ then \eqref{el1} implies that $$\sum_{i=N_0+1}^{N_1}d_ip_i<\sum_{i=N_0+1}^{N_1}p_i$$ therefore there exists a smallest $J$ between $N_0+1$ and $N_1$ such that $d_J\leq 0$. If $J<N_1$ then $$\sum_{i=N_0+1}^{N_1}d_ip_i\leq \sum_{i=N_0+1}^{N_1}p_i-p_J=p_{N_0}-p_J\leq p_{N_0}-p_{N_1-1}.$$
If otherwise $J=N_1$ and $d_J=0$, we readily get \eqref{el2}. Finally if $J=N_1$ and $d_J=-1$, then  $\sum_{i=N_0+1}^{N_1}d_ip_i=p_{N_0}-2p_{N_1}$.
This concludes the base of the induction. 

Assume now $H>1$. The inequality \eqref{el1} yields in particular 
$$\sum_{i=N_0+1}^{N_{H-1}}d_ip_i<p_{N_0}$$
and, by inductive hypothesis, either $\sum_{i=N_0+1}^{N_{H-1}}d_ip_i=p_{N_0}-p_{N_{H-1}}$ or
$\sum_{i=N_0+1}^{N_{H-1}}d_ip_i\leq p_{N_0}-{\min\{p_{N_{H-1}-1},2p_{N_{H-1}}\}}$. We distinguish the latter two cases.
\begin{itemize}
    \item Let $\sum_{i=N_0+1}^{N_{H-1}}d_ip_i=p_{N_0}-p_{N_{H-1}}$. Since by \eqref{el1} we have 
    $\sum_{i=N_0+1}^{N_{H}}d_ip_i<p_{N_0}$, then $$\sum_{i=N_{H-1}+1}^{N_{H}}d_ip_i<p_{N_{H-1}}.$$
    Arguing as in the case $H=1$ (by replacing in particular $N_0$ by $N_{H-1}$ and $N_1$ by $N_H$) we deduce that either 
    \begin{equation}\label{ele2s}\sum_{i=N_{H-1}+1}^{N_{H}}d_ip_i=p_{N_{H-1}}-p_{N_H}\end{equation}
     or  
    \begin{equation}\label{ele3s}\sum_{i=N_{H-1}+1}^{N_{H}}d_ip_i\leq p_{N_{H-1}}-{\min\{p_{N_H-1},2p_{N_H}\}}.\end{equation}
    The relations \eqref{ele2s} and \eqref{ele3s} respectively imply \eqref{el2} and \eqref{el3} and we are done. 
    \item Let $\sum_{i=N_0+1}^{N_{H-1}}d_ip_i\leq p_{N_0}-{\min\{p_{N_{H-1}-1},2p_{N_{H-1}}\}}$. Since $(p_i)$ is strictly decreasing then, in view of \eqref{golden}, $-p_{N_{H-1}-1}+p_{N_{H-1}}\leq -p_{N_{H-1}+1}\leq -p_{N_H}$. Then
    \begin{align*}
        \sum_{i=N_{0}+1}^{N_{H}}d_ip_i&=\sum_{i=N_0+1}^{N_{H-1}}d_ip_i+\sum_{i=N_{H-1}+1}^{N_H}d_ip_i\\&\leq p_{N_0}-{\min\{p_{N_{H-1}-1},2p_{N_{H-1}}\}}+p_{N_{H-1}}\\&= 
        p_{N_0}-{\min\{p_{N_{H-1}-1}-p_{N_{H-1}},p_{N_{H-1}}\}}\\
        &\leq p_{N_0}-p_{N_{H-1}+1}\leq p_{N_0}-p_{N_H-1}
    \end{align*}
    and this yields \eqref{el3} hence the claim.
\end{itemize}
\end{proof}
\begin{proof}[Proof of Theorem \ref{thm1}]
We begin by showing the if part. Let $x\in[0,\sum_{i=1}^\infty p_i]$ and let $(c_i^*)$ be the greedy expansion of $x$. We want to prove that $(c_i^*)$ is optimal. If $(c_i^*)$ unique then this is trivially true. Otherwise we can assume $(c_i^*)\not=(0)^\infty,(1)^\infty$ and that $x$ admits a different expansion $(c_i)$. Since $(c_i^*)$ is greedy, then it is  lexicographically greater than $(c_i)$. 
Consider the (non-normalized error) sequences $\Theta_n^*:=x-\sum_{i=1}^n c^*_ip_i$ and $\Theta_n :=x-\sum_{i=1}^n c_ip_i$. 
Now define $d_n:=c_n-c_n^*$ for all $n\geq 1$k, note that $(d_i)\in\{-1,0,1\}^\infty$ and the identity
\begin{equation}\label{ep1}
\Theta_n^*:=x-\sum_{i=1}^nc_i^*p_i=x-\sum_{i=1}^n c_ip_i+\sum_{i=1}^nd_ip_i=\Theta_n-\sum_{i=1}^n d_ip_i.
    \end{equation}
To conclude that $(c_i^*)$ is optimal it suffices to prove that \begin{equation}\label{claim}
\Theta_n^*\leq \Theta_n \qquad \text{for all $n\geq 0$}.\end{equation} We proceed by induction on $n$. If $n=0$ then $\Theta_0^*=x=\Theta_0$ and we are done. Assume now $n\geq 1$ and, as inductive hypothesis, that
\begin{equation}\label{ih}
\Theta_m^*\leq \Theta_m \qquad \text{for all } m=0,\dots,n-1.
\end{equation}
    Since $\Theta_0^*=\Theta_0$ then there exists a largest index $N\leq n+1$ such that $\Theta_{N-1}^*=\Theta_{N-1}$. Note that this implies $\sum_{i=N}^\infty c^*_ip_i=
    \sum_{i=N}^\infty c_ip_i$.
     Since $N$ is by definition the largest index with this property, either $N=n+1$ and $\Theta_n^*=\Theta_n$ (in this case we completed the inductive step) or $N\leq n$ and $c_N^*\not= c_N$. Now let $N\leq n$ and define $x_N:=\sum_{i=1}^{N-1}p_i+\sum_{i=N}^\infty c^*_ip_i$. Using Proposition \ref{p2}, we may conclude that $1^{N-1}(c^*_{N-1+i})=11\cdots1c_Nc_{N+1}\cdots $ is the greedy expansion of $x_N$. Also $1^{N-1}(c_{N-1+i})$ is an expansion of $x_N$ and, this implies that it is lexicographically smaller that $1^{N-1}(c^*_{N-1+i})$. Since the first $N-1$ digits of the two sequences of $x_N$ are identically $1$, the lexicographic comparison yields that the tail $(c^*_{N-1+i})$ is lexicograpchically greater than $(c_{N-1+i})$. Since we remarked above that $c_N^*\not= c_N$, then $c^*_N=1$ and $c^*_N=0$. Therefore 
    \begin{equation}\label{tn}\begin{split}
        \Theta_n^*&=x-\sum_{i=1}^{N-1}c^*_ip_i-\sum_{i=N}^nc^*_ip_i=\Theta^*_{N-1}-p_N-\sum_{i=N+1}^nc^*_ip_i\\
        &=\Theta_n-p_N+\sum_{i=N+1}^nd_ip_i.\end{split}
    \end{equation}
    If $n=N$ then we trivially have $\Theta_n^*<\Theta_n$. If $n>N$, since by definition $N$ is the largest integer such that $\Theta_{N-1}^*=\Theta_{N-1}$, then by the inductive hypothesis \eqref{ih} we have $\Theta^*_m<\Theta_m$ for all $m=N,\dots,n-1$. Moreover, if $n\leq N_1:=N+k_N+1$ then $-p_N+\sum_{i=N+1}^n\leq -p_N+\sum_{i=N+1}^{N_1}p_i=0$ and we conclude $\Theta_n^*\leq \Theta_n$ also in this case. Finally consider the case $n>N_1$. Recursively define $N_0=N$ and $N_{h+1}=N_h+1+K_{N_h}$ for $h\geq 0$. Let $H$ be such that $N_H<n\leq N_{H+1}$. Using the inductive hypothesis \eqref{ih} with $m=N_h$ and the identity \eqref{tn} we obtain
    $$\sum_{i=N_0+1}^{N_h} d_ip_i=\Theta^*_{N_h-1}-\Theta_{N_h-1}+p_{N}<p_N\qquad \text{for all }h=1,\dots,H.$$ 
    The assumptions of Lemma \ref{star} are then satisfied and we get $\sum_{i=N+1}^{N_H}d_ip_i=p_N-p_{N_H}$ or $\sum_{i=N+1}^{N_H}d_ip_i\leq p_N-p_{N_H-1}$. Both cases imply the crucial estimate
    $$\sum_{i=N+1}^{N_H}d_ip_i\leq p_N-p_{N_H}.$$
    From this we deduce that 
    $$\sum_{i=N+1}^n d_ip_i\leq p_N-p_{N_H}+\sum_{i=N_{H}+1}^nd_ip_i\leq
    p_N-p_{N_H}+\sum_{i=N_{H}+1}^{N_{H+1}}p_i=p_N.
    $$
    Plugging this inequality into \eqref{tn} we deduce
    $\Theta^*_n=\Theta_n-p_N+\sum_{i=N+1}^nd_ip_i\leq \Theta_n$.
    This concludes the proof of the inductive step and consequently we obtain the claimed estimate \eqref{claim}, and the proof of the optimality of $(c^*_i)$. By the arbitrariness of the expanded number $x$, we deduce that every greedy expansion is optimal. 
    
    \medskip
    We now prove the only if part. If there is not a sequence $(k_n)$ satisfying \eqref{el1}, then there exist a smallest $n\geq 1$ for which there exists an integer $k_n\geq 1$ such that
    \begin{equation}\label{oi}
    \sum_{i=n+1}^{n+k_n}p_i<p_n<\sum_{i=n+1}^{n+k_n+1}p_i.
    \end{equation}
    By the minimality of $n$, if $n>1$ then
    \begin{equation}\label{m}
    p_m>\sum_{i=n+1}^{n+k_n+1}p_i \qquad \text{for all $m=1,\dots,n-1$}.
    \end{equation}
    To prove \eqref{m}, suppose on the contrary that $p_m\leq \sum_{i=n+1}^{n+k_n+1}p_i$ for some $m<n$. Using the monotonicity of $(p_i)$ and the left hand-side of \eqref{oi} we have
    $$p_{m+1}<p_m\leq \sum_{i=n+1}^{n+k_n+1}p_i=\sum_{i=n+1}^{n+k_n}p_i+p_{n+1+k_n}<p_n+p_{n+1+k_n}<p_{m+1}+p_{m+2}
    $$
    that is, setting $k_m=1$, we have that $m<n$ sastisfies \eqref{oi} and this contradicts the minimality of $n$.

    Let $x:=\sum_{i=n+1}^{n+1+k_n}p_i$. Then $(c_i):=0^n1^{k_n+1}(0)^\infty$ is an expansion of $x$. Denote by $(c^*_i)$ the greedy expansion of $x$. We use \eqref{m} to compute the first digits of $(c^*_i)$: the greedy algorithm yields $c_1^*\cdots c^*_n=0^{n-1}1$ -- one can also check this a posteriori using Proposition \ref{p2}. As above, denote by $\Theta^*_n:=x-\sum_{i=1}^nc^*_ip_i$ and $\Theta_n:=x-\sum_{i=1}^nc_ip_i$. Note that $\Theta_{n+1+k_n}=0$. On the other hand
    $$\Theta^*_n=x-\sum_{i=1}^nc^*_ip_i=\sum_{i=n+1}^{n+1+k_n}p_i-p_n<\sum_{i=n+1}^{n+1+k_n}p_i-\sum_{i=n+1}^{n+k_n}p_i=p_{n+1+k_n}.$$
    Using again the greedy algorithm in Proposition \ref{p1} we get $c^*_m=0$ for all $m=n+1,\dots,n+1+k_n$. This implies
    $$\Theta^*_{n+1+k_n}=\sum_{i=n+1}^{n+1+k_n}p_i-p_n>0=\Theta_{n+k+1}.$$
    Therefore $(c^*_i)$ is not optimal and this concludes the proof.
\end{proof}

\begin{proof}[Proof of Corollary \ref{cor0}]
By Theorem \ref{thm1} we have that every greedy expansion is also optimal, therefore if $x\in[0,\sum_{i=1}^\infty p_i]$ and if $(c_i)$ is an expansion of $x$, then 
we readily have for all $n\geq 1$
$$x-\sum_{i=1}^n c_i p_i\geq \Theta^*_n.$$
Furthermore if  $\bar x:=\sum_{i=1}^\infty p_i-x$, and if $(\hat c^*_i)$ is the lazy expansion of $x$, then $(1-\hat c_i^*)$ is the greedy expansion of $\bar x$. More generally if $(c_i)$ is an expansion of $x$ then $(1-c_i)$ is an expansion of $\bar x$. By above arguments
$$\bar x-\sum_{i=1}^n (1-c_i)p_i\geq \sum_{i=1}^n (1-\hat c_i^*)p_i$$
from which we deduce
$x-\sum_{i=1}^n c_ip_i\leq \sum_{i=1}^n \hat c_i^*p_i$
and this concludes the proof.
\end{proof}

\bibliographystyle{alpha}
\bibliography{biblio}
\end{document}